\documentclass[a4paper]{amsart}
\usepackage[utf8]{inputenc}
\usepackage{xcolor,amsmath,amsfonts,amsthm,amssymb}
\usepackage[mathscr]{euscript}
\usepackage{enumerate}
\newcommand{\binary}[1]{\langle #1 \rangle_2}

\newcommand{\N}{\mathbb{N}}
\newcommand{\Z}{\mathbb{Z}}

\newcommand{\QQ}{\mathscr{Q}}

\newcommand{\X}{\mathbb{X}}

\newcommand{\lang}{\mathcal{B}}

\newcommand{\lab}{\Theta}
\newcommand{\alf}{\mathcal{A}}

\newcommand{\setsep}{\:;\:}

\renewcommand{\phi}{\varphi}
\newcommand{\modu}[1]{\:(\text{mod } #1)}
\theoremstyle{plain}\newtheorem{theorem}{Theorem}[section]
\theoremstyle{plain}\newtheorem{lemma}[theorem]{Lemma}
\theoremstyle{definition}
\theoremstyle{remark}\newtheorem{remark}[theorem]{Remark}
\theoremstyle{plain}\newtheorem{corollary}[theorem]{Corollary}

\title{Mixing properties in coded systems}
\author{Jeremias Epperlein}
\address[J. Epperlein]{Institute for Analysis and Center for Dynamics, Technische Universität Dresden, Zellescher Weg 12-14, 01069 Dresden, Germany}\email{jeremias.epperlein@tu-dresden.de}

\author{Dominik Kwietniak}
\address[D. Kwietniak]{Faculty of Mathematics and Computer Science, Jagiellonian University in Krakow, ul. \L o\-jasiewicza 6, 30-348 Krak\'ow, Poland}\email{dominik.kwietniak@uj.edu.pl}
\urladdr{www.im.uj.edu.pl/DominikKwietniak/}

\author{Piotr Oprocha}
\address[P. Oprocha]{AGH University of Science and Technology, Faculty of Applied
	Mathematics, al.
	Mickiewicza 30, 30-059 Krak\'ow, Poland}
\email{oprocha@agh.edu.pl}
\date{\today}

\begin{document}

\begin{abstract}
We show that topological mixing, weak mixing and total transitivity are equivalent for coded systems.
We provide an example of a mixing coded system which cannot be approximated by any
increasing sequence of mixing shifts of finite type, has only periodic points of even period and each set
of its generators consists of blocks of even length. We prove that such an example cannot be a synchronized system.
We also show that a mixing coded systems has the strong property
$P$. 
\end{abstract}
\maketitle

\section{Introduction}
We consider coded systems and their recurrence properties that are stronger than topological transitivity. We are interested
in topological (weak) mixing, and properties like the strong property $P$, which is a variant of the specification property.

Recall that a shift space is a \emph{coded system} if it can be presented by an irreducible directed graph whose edges are labelled by symbols from a finite alphabet $\alf$.
Here, ``presented'' means that the shift space is the closure in $\alf^\Z$ of all bi-infinite sequences of symbols which are labels for bi-infinite paths in the graph. Equivalently, $X$ is a coded system if it is a closure in $\alf^\Z$ of the set of all bi-infinite sequences obtained by freely concatenating the words in a (possibly infinite) list of words over $\alf$. Such a list is the set of \emph{generators} of $X$.

Transitive sofic shifts are coded systems which can be presented by a finite irreducible digraph. Coded systems were introduced by Blanchard and
Hansel \cite{BH}, who showed that any factor of a coded
system is coded. This is a generalization of a well-known property of sofic shifts \cite[Corollary 3.2.2]{LM}.

It is natural to ask which properties of irreducible sofic shifts extend
to coded systems. We address an aspect of this problem which leads to an extension of known results and a quite unexpected example.
The counter-intuitive nature of this example is in our opinion the most interesting feature of this paper, but we hope that our other results
 will fill a gap in the literature. Furthermore, there has been a resurgence in interest in coded systems in general,
and their notable subclasses in particular (for example $S$-gap shifts \cite{BG,CT12,DJ}, $\beta$-shifts \cite{CT12}, Dyck shifts \cite{M08, M13}). Coded systems often provide a testing ground for further extensions (see \cite{CT14},   
where the line of investigation initiated in \cite{CT12} is developed and extended to non-symbolic systems). Therefore understanding the situation for coded systems may lead to a solution of more general problems.

In order to describe our results, note first that any coded system is topologically transitive (irreducible).
Recall also that for a non trivial transitive sofic shift $X$ the following stronger variants of transitivity
are equivalent:
\begin{enumerate}[(a)]
\item\label{list:spec} $X$ has the periodic specification property;
\item\label{list:mix} $X$ is topologically mixing;
\item\label{list:prop-p} $X$ has the strong property $P$;
\item\label{list:weak-mix} $X$ is topologically weakly mixing;
\item\label{list:tot-trans}$X$ is totally transitive;
\item\label{list:rel-prime} $X$ has two periodic points with relatively prime primary periods.
\end{enumerate}
It seems that this is a folklore theorem. It follows from various
well-known results, but we could not find it in this form in the
literature.

Here we examine the connections between transitivity variants \eqref{list:mix}--\eqref{list:rel-prime} for non necessarily sofic coded systems. Note that the proof of equivalence of properties \eqref{list:mix}--\eqref{list:rel-prime}  for shifts of finite type can be adapted for irreducible Markov shifts over countable alphabets. Furthermore, every coded system contains a dense subset which is a factor of an irreducible countable Markov shift. This suggests that properties \eqref{list:mix}--\eqref{list:rel-prime} should remain equivalent for coded systems. And this is indeed the case with one notable exception.

It is easy to see that property \eqref{list:rel-prime} implies topological mixing for coded systems. We prove, somewhat surprisingly, that the converse does not hold for all coded systems. We construct a mixing coded systems without a periodic point of odd period and hence without a generator of an odd length. Such a system cannot contain a mixing shift of finite type. Note that Krieger \cite{K} characterized coded systems as those shift spaces which contain an
increasing sequence of irreducible shifts of finite type with dense union. Krieger's characterization is the best possible, in the sense that there exists an increasing sequence of sofic shifts whose closure is a shift space which is not a coded system \cite{BH2}.
It follows from our result that there are mixing coded systems, which cannot be approximated from the inside by mixing shifts of finite type as they do not contain any mixing sofic shift.

On the other hand, \eqref{list:mix}--\eqref{list:tot-trans} are still equivalent for coded systems. We also show that if $X$ is a transitive but not totally transitive coded system, then for some prime $p$ we can write $X=X_1\cup \ldots X_p$, where $X_j$'s are closed subsets of $X$ cyclically permuted by the shift map $\sigma$, each $X_j$ is $\sigma^p$ invariant, $X_i\cap X_j$ is nowhere dense for $i\neq j$, and $(X_j,\sigma^p)$ is topologically mixing. Finally, we note that \eqref{list:spec} implies \eqref{list:mix} by the definition, but is not implied by any of the conditions \eqref{list:mix}--\eqref{list:rel-prime}, since all $\beta$-shifts are mixing but some do not have the specification property (see \cite{Tho12}).

This paper is organized as follows: In the next section we set up the notation and terminology. In Section 3 we prove a structure theorem for topologically mixing coded systems  and that total transitivity, weak mixing and mixing are equivalent for coded systems. Section 4 contains an example of a shift space which has the strong property $P$, but is not topologically mixing. In Section 5 we describe a mixing coded systems without a periodic point of odd period. In Section 6 we establish an equivalence of the strong property $P$ and topological mixing for coded systems. In section 7 we present some complementary results on synchronized systems. They imply that \eqref{list:mix}--\eqref{list:rel-prime} are equivalent for synchronized systems, thus our example does not have any synchronizing words.

\section{Notation and definitions}

We assume the reader is familiar with elementary symbolic dynamics as in \cite{LM}.
We fix a finite set $\alf$ with at least two elements and call it the \emph{alphabet}.
Let $\alf^\Z$ denote the set of bi-infinite (two-sided) sequences
\[
x = (x_i)_{i\in\Z} = ...x_{-3}x_{-2}x_{-1}x_0x_1x_2\ldots,
\]
such that $x_i \in\alf$ for all $i$. We equip $\alf$ with the discrete topology and we
consider $\alf^\Z$ as a compact metric space in the product topology. The shift operator on $\alf^\Z$ is denoted by $\sigma$.
A \emph{shift space} over the alphabet $\alf$ is a shift-invariant subset of $\alf^\Z$
which is closed in that topology. The set $\alf^\Z$ itself is a shift space
called the \emph{full shift}. In this paper all shift spaces will be two-sided and transitive.

A \emph{block} of length $k$ is an element $w = w_1w_2\ldots w_k$ of $\alf^k$. Throughout this paper ``\emph{a word}'' is a synonym for ``\emph{a block}''. The length of a block is denoted $|w|$. The set of all words over $\alf$ is denoted by $\alf^*$.
Given $x \in\alf^\Z$ and $i,j\in\Z$ with $i\le j$ we write $x_{[i,j]}$ to denote the block $x_ix_{i+1}\ldots x_j$. We
say that a block $w$ occurs in $x$ if $w = x_{[i,j]}$ for some $i,j\in\Z$. A language of a shift space $X$ is the set $\lang(X)$ of all blocks that occur in
$X$. The set of blocks of length $n$ in the language of $X$ is denoted $\lang_n(X)$. Similarly, the set of all words that occur in a point $x\in\alf^\Z$ is denoted $\lang(x)$.
The \emph{empty word} $\perp$ is the unique word of length $0$. We write $\alf^+$ for the set of nonempty words over $\alf$.

A \emph{central cylinder set} of a word $u\in\lang_{2r+1}(X)$, where $r\in\N$, is the set $[u]\subset X$ of points from $X$ in which the block $u$ occurs starting at position $-r$, that is, $\{y\in X: y_{[-r,r]}=u\}$. Central cylinders (or \emph{cylinders} for short) are open and closed subsets of $X$. The family
\[
\{[x_{[-r,r]}]: r\in\N\}
\]
of cylinder sets determined by a central subblock of $x$ is a neighbourhood basis for a point $x\in X$.
We use a multiplicative notation for \emph{concatenation} of words, so that $w^n=w\ldots w$ ($n$-times) and $w^\infty=www\ldots\in\alf^\N$.

Given a set of words $\mathscr{Q}\subset\alf^+$, we define $\mathscr{Q}^0=\{\perp\}$, and
$\mathscr{Q}^n=\mathscr{Q}^{n-1}\mathscr{Q}=\{uw: u\in\mathscr{Q}^{n-1},\,w\in \mathscr{Q}\}$.
We also let $\mathscr{Q}^+$ denote the set of all possible finite concatenations of words from $Q$, that is, $\mathscr{Q}^+=\bigcup_{n=1}^\infty\mathscr{Q}^{n}$. In particular, $\mathscr{Q}\subset \mathscr{Q}^+$.

By $\mathscr{Q}^\Z$ we denote the set containing all possible bi-infinite concatenations of elements of $\mathscr{Q}$, that is, $x\in \mathscr{Q}^\Z$ if $x$ can be partitioned into elements of $\mathscr{Q}$.

By a \emph{countable graph} we mean a directed graph with at most countably many vertices and edges.
A countable graph $G$ is \emph{labeled}  if there is a \emph{labeling} $\Theta$ which is simply a function
from the set of edges of $G$ to the alphabet $\alf$. A labeling of edges extends, in an obvious way, to a labeling of all finite (respectively, infinite, bi-infinite) paths on $G$ by blocks (respectively, infinite or bi-infinite sequences) over $\alf$.
The set $Y_G$ of bi-infinite sequences constructed by reading off labels along a bi-infinite path on a labeled graph $(G,\Theta)$ is
shift invariant, but usually it is not closed and therefore not a shift space. Nevertheless,
its closure $X=\overline{Y_G}$ in $\mathcal{A}^{\mathbb{Z}}$ is a shift space, and we say that $X$ is \emph{presented} by $(G,\Theta)$. Any shift space admitting such a presentation is a \emph{coded system}. A set of \emph{generators} for a shift space $X$ is a family of words $\QQ\subset \alf^*$ such that
the language of $X$ coincides with the set of all subblocks occurring in elements of $\QQ^+$. Equivalently, $X$ is the closure of $\QQ^\Z$ in $\alf^\Z$. Every coded systems has a set of generators, and conversely, if a transitive shift space has a set of generators then it is a coded system.
A countable graph is \emph{irreducible} if given any pair of its vertices, say $(v_i, v_j)$,
there is a path from $v_i$ to $v_j$.

%

Dynamical properties like those mentioned in \eqref{list:spec}--\eqref{list:tot-trans} above are usually defined for a continuous map acting on a  metric space. Here we define them in the language of symbolic dynamics.

A shift space $X$ is:
\begin{enumerate}
  \item \emph{transitive} if for any $u,v\in \lang(X)$ there is $w\in \lang(X)$ such that $uwv\in \lang(X)$;
  \item \emph{totally transitive} if for any $u,v\in \lang(X)$ and any $n>0$ there is $w\in \lang(X)$ such that $uwv\in \lang(X)$ and $n$ divides $|uw|$;
  \item \emph{weakly mixing} if for any $u_1,v_1,u_2,v_2\in \lang(X)$ there are $w_1,w_2\in \lang(X)$ such that $u_1w_1v_1,u_2w_2v_2\in \lang(X)$ and $|u_1w_1|=|u_2w_2|$;
  \item \emph{mixing} if for every $u,v\in \lang(X)$ there is $N>0$ such that for every $n>N$ there is $w\in \lang_n(X)$ such that $uwv\in \lang(X)$.
\end{enumerate}
We say that a shift space has:
\begin{enumerate}
  \item the \emph{strong property $P$} if for any $k\ge 2$
 and any words $u_1,\ldots, u_k\in \lang(X)$ with $|u_1|=\ldots=|u_k|$ there is an $n\in\N$ such that for any $N\in\N$ and function
 $\varphi\colon\{1,\ldots,N\}\to\{1,\ldots,k\}$ there are
 words $w_1,\ldots, w_{N-1}\in \lang_n(X)$ such that $u_{\varphi(1)}w_1u_2\ldots u_{\varphi(N-1)}w_{N-1}u_{\varphi(N)}\in \lang(X)$;
  \item the \emph{specification property} if there is an integer $N\geq 0$ such that for any $u,v\in \lang(X)$ there is $w\in \lang_N(X)$ such that $uwv\in \lang(X)$.
\end{enumerate}
Blanchard \cite{B92} proved that the strong property $P$ implies weak mixing, and does not imply mixing.

It is convenient to rephrase  the above definitions using the sets
\[
N_\sigma([u],[v])=\{\ell\in\N: uwv\in\lang(X)\text{ for some }w\in\alf^* \text{ such that }|uw|=\ell\},
\]
where $u,v\in\lang(X)$. For example a shift space $X$ is weakly mixing if  for every $u,v\in\lang(X)$ the set $N_\sigma([u],[v])$ contains arbitrarily long intervals of consecutive integers (see \cite[Theorem 1.11]{G}). A shift space $X$ is transitive for $\sigma^k$ where $k\in\N$ if and only if
the set
\[
N_{\sigma^k}([u],[v])=\{\ell\in\N:  uwv\in\lang(X)\text{ for some }w\in\alf^* \text{ such that }|uw|=k\ell\},
\]
is non-empty for every $u,v\in\lang(X)$.

If a dynamical systems on a compact metric space is transitive, then there is a dense $G_\delta$-set of points with dense orbit. In particular, in a transitive shift space $X$ in every cylinder set there is a point $x$ such
that every block in $\lang(X)$ occurs infinitely many times in $x$.

A \emph{synchronizing word} for a shift space $X$ is an element $v$ of $\lang(X)$ such that $uv,vw\in\lang(X)$ for some blocks $u,w$ over $\alf$ imply $uvw\in\lang(X)$. A \emph{synchronized system} is a shift space with a synchronizing word. Synchronized systems were introduced in \cite{BH}. Every synchronized system is coded. The uniqueness of the minimal right-resolving presentation known for sofic shifts
extends to synchronized systems as outlined in \cite[p. 451]{LM} (see  also \cite[p. 1241]{Thomsen} and references therein). Synchronized systems and their generalizations were extensively studied in \cite{FF}.

Let $x_1, x_2, \ldots, x_n$ be positive integers. It is well-known that every sufficiently large integer can be represented as a non-negative integer linear combination of the $x_i$ if and only if $\gcd(x_1,x_2,\ldots,x_n)=1$. For a later reference we formulate an important consequence of this result as a remark.

\begin{remark}\label{rem:frob}
Let $x_1, x_2, x_3,\ldots$ be positive integers. If $\gcd(x_1,x_2,x_3,\ldots)=k$, then every sufficiently large multiple of $k$ can be represented as a non-negative integer linear combination of the $x_i$.
\end{remark}

\section{Total transitivity implies mixing}

We prove  that total transitivity, weak mixing and mixing are equivalent for coded systems.
This leads to a structure theorem for coded systems which are not totally transitive.

\begin{theorem}\label{thm:mpdcoded}
Suppose that $X$ is a coded system and let $D\subset X$ be a closed set with nonempty interior such that $\sigma^k (D)\subset D$ for some $k>0$. If the shift space $(D,\sigma^k)$ is totally transitive, then it is mixing.
\end{theorem}
\begin{proof}
Let $G$  be an irreducible countable labeled graph presenting $X$. Since $D$ has nonempty interior, for each $r\in\mathbb{N}$ large enough there is $w\in \lang(X)$ of length $2r+1$ such that the cylinder $[w]=\{x\in X : x_{[-r,r]}=w\}$ is contained in the interior of $D$.

We claim that the periodic points of $(D,\sigma^k)$ are dense in $D$. Let $V$ be a nonempty open subset of $D$.
Then there is an open set $U\subset X$ with $V=U\cap D$. Without loss of generality we may assume that $U$ is a cylinder set of some block $u\in\lang(X)$. By total transitivity there is $s\in\lang(X)$ and a path on $G$ labeled by $wsu$ such that the length of $ws$ is $jk$ for some integer $j$. We can join the last vertex on the path labeled by $wsu$ with the first vertex of the same path by a path labeled by a word $t$. Therefore there is
a bi-infinite periodic sequence $y\in X$ such that $y_{[-r,\infty)}=(ws ut)^\infty$ for some words $s,t\in\lang(X)$. But then $y\in [w]\subset D$
and $\sigma^{jk}(y)\in [u]$ because $|ws|=jk$. On the other hand $y\in D$, hence $\sigma^{jk}(y)\in D$ and then $\sigma^{jk}(y)\in U\cap D=V$. This shows that periodic points are dense in $D$ under $\sigma^k$. Since every totally transitive dynamical system with dense set of periodic points is weakly mixing (e.g. see \cite{BanksWM}), it follows that $(D,\sigma^k)$ is weakly mixing.

We proceed to a proof of mixing of $(D,\sigma^k)$.
We recall that by \cite[Lemma 3.1]{HKO} a shift system $(D,\sigma^k)$ is mixing if for each cylinder $v$ in some neighborhood basis of a point with dense orbit in $D$ the set $N_{\sigma^k}([v],[v])$ is cofinite.
Let $\bar{x}\in[w]$ be a point whose orbit is dense in $D$ under $\sigma^k$.
Given $a\ge r$ set $v_a=\bar{x}_{[-a,a]}$ and
\[
[v_a]=\{x\in D : x_{[-a,a]}=v_a\}.
\]
Note that $\bar{x}\in[v_a]\subset [w]\subset D$, hence the cylinder of $v_a$ in $X$ and in $D$ coincide.
Furthermore,  $([v_a])_{a \geq r}$ is a neighborhood basis of $\bar{x}$.

It is now enough to show that the set $N_{\sigma^k}([v_a],[v_a])$ is cofinite.  Let $u$ be a word such that $v_au$ is a labeling of a loop in $G$ and let $m$ be the length of $v_au$. Without loss of generality we may assume that $k$ divides $m$ (we replace $vu$ by $(v_au)^k$ if necessary). Denote the loop presenting $v_au$ on $G$ by $\xi$. By weak mixing of $(D,\sigma^k)$ the set $N_{\sigma^k}([v_a],[v_a])$ contains a set of $m$ consecutive integers, hence there is an integer $q>0$ such that for each $i=1,\ldots, m$ the graph $G$ contains a path $\eta_i$ labeled $v_a v_i v_a$ where $|v_i|=(q+i)k-|v_a|$.  Since $G$ is irreducible, for each $i=0,1,\ldots, m$ there exists a path $\gamma_i$ in $G$ such that the following path is a loop on $G$:
\[
\pi=\xi \gamma_0 \eta_1 \gamma_1\eta_2\ldots \gamma_{m-1} \eta_m\gamma_m.
\]
Let $p=|\pi|$. We claim that for every $j\ge 1$ and $i=1,\ldots,m$ we have
\[
	p+(mj)+(q+i)k \in N_{\sigma}([v_a],[v_a]).
\]
In order to show this, consider the labeling of the following path:
\[
\eta_i \gamma_i\ldots \eta_m \gamma_m (\xi)^j \gamma_0 \eta_1 \gamma_1 \ldots \gamma_{i-1}\eta_i.
\]
It starts and ends with $v_a$ (it is a path on $G$ because $\pi$ and $\xi$ are loops). 
This proves that $N_{\sigma^k}([v_a],[v_a])$ is cofinite.
\end{proof}

\begin{corollary}\label{cor:equiv}
If a coded system $X$ is totally transitive, then it is mixing.
\end{corollary}
\begin{proof}
Take $D=X$ and apply Theorem \ref{thm:mpdcoded}.
\end{proof}

We will now describe the structure of coded systems which are not totally transitive.
Banks proved in \cite{BanksWM} that if a dynamical system $(X,T)$ is transitive, but $(X,T^k)$ is not transitive for some $k>1$ then there is a \emph{regular periodic decomposition of $X$}, that is, one can find a finite cover $\{D_0,\dots, D_{k-1}\}$ of $X$ by non-empty regular closed sets with pairwise disjoint interiors such that $T(D_{i-1})\subseteq D_{i \modu{k}}$ for each $1\le i\le k$.  In this case we say that $k$ is the \emph{length} of the decomposition. Recall that a set is \emph{regular closed} if it is the closure of its interior. If $(D_i,T^n)$ is topologically mixing for some regular periodic decomposition $\mathscr{D}=\{D_0,\dots, D_{k-1}\}$ of $X$ then we say that $T$ is \emph{relatively mixing} with respect to $\mathscr{D}$. Observe that $(X,T)$ is relatively mixing with respect to $\mathscr{D}$ if and only if $(D_0,T^n)$ is topologically mixing. It was proved in \cite{BanksWM} that if there is an upper bound on the possible lengths of periodic decompositions of a transitive dynamical system then there exists a regular periodic decomposition $D_0,D_1,\ldots, D_{n-1}$ such that $(D_i, T^n)$ is totally transitive for every $0\le i <n$ (this decomposition is called \emph{terminal}).

\begin{theorem}\label{thm:rel-mix}
Every coded system is relatively mixing.
\end{theorem}
\begin{proof}Let $X$ be a coded system.
By Theorem~\ref{thm:mpdcoded} and the result of Banks mentioned above, it suffices to show that there is an upper bound on the possible lengths of regular periodic decompositions of $X$.

Let $G$  be an irreducible countable labeled graph presenting $X$ and let $k$ be the length of a cycle $\eta$ in $G$. We claim that the length of a periodic decomposition of $X$ can not be greater than $k$. On the contrary, assume that $D_0,\ldots, D_{n-1}$ is a regular periodic decomposition and $n>k$. Since $D_0$ is regularly closed, there is $r\in\N$ and a word $w$ of length $2r+1$ such that the cylinder $[w]=\{x : x_{[-r,r]}=w\}\subset D_0$. Since for each $i>0$ the interior of a regular closed set $D_i$ is disjoint  with the interior of $D_0$, we have $D_i \cap [w]=\emptyset$ for each $i>0$.

Since $G$ is irreducible, there are paths $\pi,\gamma$ with $|\pi|\geq |w|$ such that $\pi \eta \gamma$
is a cycle on $G$ labeled $wu\in \lang(X)$ for some word $u$. Repeating the path $\pi\eta \gamma$ if necessary we may assume that $n$ divides $|\pi\eta\gamma|=|wu|$. Let a word $wu'\in \lang(X)$ be the label of the path $\pi \eta \eta \gamma$ on $G$. We have $|wu'|=nj +k$ for some $j\geq 1$. Note that $wu'w\in \lang(X)$ because $\pi \eta \eta \gamma \pi$ is a path on $G$
and hence there is $x\in X$ with $x_{[-r,t]}=wu'w$ for some $t>r$. But then $x$ and  $\sigma^{nj+k}(x)$ both belong to $[w]\subset D_0$. On the other hand
\[\sigma^{nj+k}(x)\in \sigma^{nj+k}(D_0)=\sigma^k(D_0)=D_k.\]
Since $k<n$, we have
$[w]\cap D_k =\emptyset$ which leads to a contradiction.
\end{proof}

\section{Property $P$ does not imply mixing}

We construct a weakly mixing but not mixing spacing shift $Y$ with the strong property $P$. This shows that the property $P$ and topological mixing are not equivalent in general. Note that $Y$ can not be coded system by Corollary \ref{cor:equiv}. A similar example was first given by Blanchard \cite{B92}, but our construction is much simpler.

Given $R\subset\mathbb{N}$ we define a \emph{spacing shift} $\Omega_R$ as the set of all $x\in\{0,1\}^\mathbb{Z}$ such that the condition $x_i=x_j=1$ for some $i,j\in\mathbb{Z}$ with $i\neq j$ implies $|i-j|\in R$. Elements of $\lang(\Omega_R)$ are called \emph{$R$-allowed blocks} (see \cite{Banks-etal,LZ} for more details).

\begin{theorem}
There is a non-mixing shift space $Y$ with the strong property $P$.
\end{theorem}
\begin{proof}
We construct a \emph{spacing shift} with the desired properties.  Below we write $\binary{n}$ for the binary representation of a positive integer $n$, that is,
\[
\binary{1}=1,\, \binary{2}=10,\,\binary{3}=11,\ldots.
\]
Also for $u=u_1\ldots u_n\in\{0,1\}^n$ we let $\Delta(u)=\{|i-j|:1\le j<i\le n,\, u_i=u_j=1\}$.

Let $R=\N\setminus\{2^k:k\in\N\}$. Define $Y=\Omega_R$, and note that $R$ is thick, thus $\Omega_R$ is nontrivial and weakly mixing (see \cite{LZ}). We claim that for every $k\in\N$, $L=2^k$, $w=0^{2L}$ and any family of $R$-allowed blocks $v_0,v_1,\ldots,v_t$ of length $L$,  the block  $u=v_0wv_1wv_2\ldots v_{t-1}wv_t$ is also $R$-allowed. This clearly implies that the spacing shift $\Omega_R$ has the strong property $P$.
A simple calculation yields that
\[
\Delta(u) \subset \bigg(\{1,\ldots,L-1\}\cap R \bigg)\cup\bigg(\bigcup_{m=0}^\infty \{(3m+2)L+1,\ldots,(3m+4)L-1\}\bigg).
\]
It is enough to show that no power of $2$ is in $\Delta(u)$.

Note that any
\[q\in \bigcup_{m=0}^\infty \{(3m+2)L+1,\ldots,(3m+4)L-1\}\]
can be written as $q=a+b$ where $a\in \{2^{k+1}+1,\ldots,2^{k+2}-1\}$ and $b=3m\cdot 2^k$.
If $b=0$ then clearly $q=a$ is not a power of $2$, hence we may assume that $m>0$.
Then $\binary{b}=\binary{3m} 0^k$ and
$\binary{a}=1x_k\ldots x_{0}$,  where not all $x_i$'s are $0$.
Denote
\[
\binary{a+b}=y_ly_{l-1}\ldots y_k\ldots y_0.
\]
Note that $l>k+1$.
If $x_i\neq 0$ for some $i=0,1,\ldots,k-1$, then $y_i\neq 0$ and $a+b$ is not a power of $2$.
If $x_0=x_1=\ldots=x_{k-1}=0$ and $x_k=1$, then $a=3\cdot 2^k$. In that case, $a+b$ is also divisible by $3$ and hence it is not a power of $2$.
Therefore $\Delta(u)\subset R$ and $\Omega_R$ has the strong property $P$.

On the other hand $\Omega_R$ is not topologically mixing because it is easy to see that $R=N([1]_R,[1]_R)$ is not cofinite which is a necessary condition for topological mixing (see \cite{Banks-etal}, cf. \cite{LZ}). Here
$[1]_R=\{x\in\Omega_R:x_0=1\}$ is a nonempty open subset of $\Omega_R$.
\end{proof}

\section{A mixing coded system without a generators of coprime length}

We construct a mixing coded system without periodic points of odd period and such that every set of generators for this system contains only words of even length.

Let $t=t_0t_1t_2\ldots=10010110\ldots$ be the Prouhet-Thue-Morse sequence (see \cite{Allouche}). Recall that it obeys $t_{2n}=t_n$ and $t_{2n+1}=1-t_n$. It is well-known that $t$ is a cube-free sequence, hence neither $000$ nor $111$ occur in $t$. Furthermore, $\lang(t)$ is a language of a minimal and non-periodic shift space $X_{\text{TM}}$.

We first define auxiliary sets $L_n\subset\{0,1\}^*$ for $n=1,2,\ldots$ and a sequence of words $\{a_k\}_{k=0}^\infty$.
We begin by setting $a_0=01$ 
and $L_1:=\{a_0\}$.
Assume that we have performed $n-1$ steps of our construction ($n\in\N$). We are given the set $L_{n-1}$ and $\{a_k\}_{k=0}^\infty$ is defined for indices $0,1,\ldots,s_n-1$, that is, $s_n$ denotes the number of words in the sequence $\{a_k\}$ constructed up to the step $n$.  In particular, we have $s_1=0$ and $s_2=1$.
At each step $n\ge 2$ we enumerate the blocks in $L_{n-1}$ starting from $s_n$, that is, we write
\[
L_{n-1}=\{w_{s_n},\dots,w_{s_n+|L_{n-1}|-1}\}. 
\]
We extend the sequence $\{a_k\}$ by adding words
\[
a_{j}=01110 t_{[0,4j-3]}011110 w_{j} 011110 t_{[0,4j-1]} 01110.
\]
for $j=s_n,\ldots, s_n+|L_{n-1}|-1$.
Then we set
\begin{align*}
L_n:=& \bigg\{a_{s_n}, a_{s_n+1},\ldots, a_{s_{n+1}-1}\bigg\}\cup\bigcup_{k=1}^n L_n^k,
\end{align*}
where $L_n^k=\{w_1w_2\ldots w_k:w_j\in L_n\text{ for } j=1,\ldots,k\}$. This completes the step $n$ and our induction.
Let $\QQ:=\{a_i \setsep i \in \N_0\}$. 

We will call the words $01110$ and $011110$  \emph{markers}. Note that $a_0$ is the only element of $\QQ$ without markers, and since $111$ is never a subblock of $\lang(t)$ we can identify positions of all markers in $a_k$ and therefore we can identify also positions of blocks $t_{[0,4j-3]}$ and $t_{[0,4j-1]}$. Hence knowing that $w\in\QQ$ and the length of the longest subblock from $\lang(t)$ in $w$  between two markers (when $w\neq a_0$) we can uniquely determine $k$ such that $w=a_k$.

Notice that
\begin{equation}
L_n \subset L_{n+1}\quad \text{ and }\quad \QQ^+=\bigcup_{n=1}^{\infty} L_n,
\label{eq:propofG}
\end{equation}
thus $\QQ$ and $\bigcup_{n=1}^\infty L_n$ generate the same coded system denoted by $\X$. 

\begin{lemma}\label{lem:mixingXcp}
The coded system $\X$ is mixing.
\end{lemma}
\begin{proof}
Every block $u\in\lang(\X)$ is a subword of some concatenation of generators. Therefore
it is enough to show that for any $k,\ell \in \N$ and $u_1,\dots,u_k, v_1\dots v_\ell \in \QQ$
there is $M \in \N$ such that for all $m >M$ there is a block $w \in \{0,1\}^m$ with
$u_1\dots u_\ell w v_1 \dots v_k \in \lang(\X)$.

Observe that by \eqref{eq:propofG}  there is $n \in \N$ such that
$u_1,\ldots,u_k \in L_{n}$.  Clearly, we may also assume that $n>k$.
Then it follows directly from the construction that $u=u_1\dots u_k\in L_{n+1}$ and there is $q \in \N$
such that
\[
a_{q}=01110 t_{[0,4q-3]} 011110 u 011110 t_{[0,4q-1]} 01110\in L_{n+2}.
\]
Define $s_{2q-1}=t_{[0,4q-3]}$ and $s_{2q}=t_{[0,4q-1]}$.
Set $M=4q+11$. If $m>M$ is odd, then we have
\begin{multline*}
a_q a_0^{\frac{m-4q-11}{2}}v_1\dots v_k=
\\
01110 s_{2q-1} 011110 \underbrace{u_1\ldots u_\ell}_{u} \underbrace{011110 s_{2q} 01110 (01)^{\frac{m-4q-11}{2}}}_{w\in\{0,1\}^m}{v_1\dots v_k}
\in \QQ^+.
\end{multline*}

For even $m>M$ we have
\[
ua_0^{\frac{m}{2}}v_1\dots v_k=u(01)^{\frac{m}{2}} v_1\dots v_k \in \QQ^+.
\]
This completes the proof, since $\QQ^+\subset \lang(\X)$.
\end{proof}

\begin{lemma}\label{lem:generators}
  If $X$ is a non trivial coded system generated by a set $Q$ and there is a word $w\in Q$ with odd length, then $X$ contains a
  periodic point of odd prime period greater than one.
\end{lemma}
\begin{proof}
Since $X$ is nontrivial, there is a word in $\lang(X)$ containing symbol $0$ and a word containing symbol $1$. Hence there is a non constant word $u$ in $Q^+$, thus $uu\in Q^*$ is a non-constant word of even length. Then the word $uuw$ is a non constant word of odd length $k$ in $Q^+$ whose infinite concatenation is a non-constant periodic point with an odd prime period dividing $k$.
\end{proof}

We are going to prove that $\X$ has no periodic points with odd period. 
We first show that 
an odd periodic point cannot occur in one of the sofic shifts $Y_n$ generated by $L_n$.
\begin{lemma}\label{lem:primep1}
  For every $n \in \N$ the sofic shift $Y_n$ generated by the set $L_n$ does not contain a periodic point with odd prime period.
\end{lemma}
\begin{proof}
Fix any $n\in \N$ and let be a point $x\in Y_n$ with
prime period $q$. Clearly, $q>1$ because the lengths of runs of $0$'s and $1$'s in $\lang(\X)$ are bounded.
If $x$ does not contain any marker then $x=(01)^\infty$ and $q$ is even.
Thus we may assume that there are (infinitely many) markers in $x$.
Let $\ell$ be the length of the longest block $w$ from $\lang(t)$ appearing in $x$ between two markers.
Then $\ell=4k$ for some $k$, and there must be $j \in \Z$ such that $x_{[j,j+|a_k|-1]}=a_k$  (no word $a_r$ with $r>k$ can appear in $x$, since then we would have $\ell\geq 4k+4$). Because $x$ has period $q$, we also have $x_{[j+q,j+q+|a_k|-1]}=a_k$. Let $w\in \QQ^+$ be a word which contains $x_{[j,j+q+|a_k|-1]}$ as a subblock and which does not contain
a block from $\lang(t)$ of length greater then $4k$. Write $w=v_1\ldots v_t$ where $v_j\in \QQ$. By the above observation, we have that $x_{[j,j+|a_k|-1]}=x_{[j+q,j+q+|a_k|-1]}=a_k\in \QQ$ must be among $v_i$.
This immediately implies that $x_{[j,j+q-1]} \in \QQ^+$ and since all words in $\QQ$ have even length, we find that
$q=|x_{[j,j+q-1]}|$ is even.
\end{proof}

Finally we show that taking the closure of $\bigcup_{n=1}^\infty Y_n$ does not introduce periodic points.
\begin{lemma}\label{lem:primep2}
  Any element $x$ of $\X \setminus \bigcup_{n=1}^\infty Y_n$ contains arbitrary long blocks from $\lang(t)$.
  In particular, $x$ cannot be periodic.
\end{lemma}
\begin{proof}
Assume on the contrary that there is $x\in \X \setminus \bigcup_{n=1}^\infty Y_n$ such that the longest block from $\lang(t)$ appearing in $x$ has length at most $4k$ for some integer $k>0$. 
Then $x$ must contain infinitely many markers as subwords, as otherwise $x=\ldots a_0a_0 w a_0 a_0 \ldots$ for some $w\in L_n$ and some $n$, thus $x\in Y_n$.

There exists an infinite set $J\subset \Z$
and a strictly increasing infinite sequence of integers $(n_i)_{i \in J}$ such
that $x_{[n_i,n_i+4]}=01110$ if and only if there is $i\in J$ such that $j=n_i$.
    Let $m$ be the least integer such that $a_{k} \in L_m$.

 We claim that every word $w_i:=x_{[n_{-i},\dots,n_i+4]}$ is contained in $\lang(Y_m)$.
  Let $j$ be the smallest positive integer, such that $w_i \in \lang(Y_j)$.
  There must be some word $g \in L_j^+$ such that $g=bw_ic$ with $b,c \in \lang(\X)$.
  Either $j \leq m$, in which case we are done, or $j>m$ and $g$ must contain $a_\ell$ for some $\ell >k$.
  Since $w_i$ starts and ends with $01110$, the two longest blocks of symbols from $\lang(t)$ occurring  in $a_\ell$ must
  be contained in $b$ and $c$ and thus $w_i$ is already contained in the middle word
  of $a_\ell$, which is an element of $L_{j-1}$. This contradicts the minimality of $j$.
Therefore our claim holds.
If $J$ is bi-infinite, then $x \in Y_m$ which is a contradiction.
Otherwise, either
\begin{eqnarray*}
a&=&\inf\{n_i : i\in J\}=\min\{n_i : i\in J\}>-\infty, \quad \text{or}\\
b&=&\sup\{n_i : i\in J\}=\max\{n_i : i\in J\}<\infty.
\end{eqnarray*}
If $a >-\infty$, then $x_{(-\infty, a-1]}$ does not contain markers and $x_{[a, a+4]}=011110$ which implies that $x_{(-\infty, a-1]}=\ldots a_0 a_0 a_0$.
In the second case $b<\infty$ and we obtain that $x_{[b+1,\infty)}=a_0 a_0 a_0\ldots $,
hence both cases imply that $x \in Y_m$ completing the proof.
\end{proof}

 \begin{theorem}
 There exists a coded system $X$ which is mixing, but does not have periodic points with odd periods. In particular, every set of generators for $X$ contains only blocks of even length.
 \end{theorem}
 \begin{proof}
 The shift $X=\X$ is mixing by Lemma~\ref{lem:mixingXcp}. Combining Lemma~\ref{lem:primep1} and \ref{lem:primep2} we see that $\X$ does not contain a periodic point with odd prime period. Then it follows from Theorem~\ref{lem:generators} that every set of generators for $X$ contains only blocks of even length.
 \end{proof}


\section{Strong property $P$ and mixing coded systems}

As we have seen before, the strong property $P$ does not imply topological mixing.
Clearly, the converse implication is neither true, since Blanchard \cite{B92} proved that the property $P$ implies positive topological entropy and there are examples of mixing shifts with zero topological entropy.

The purpose of this section is to show that every mixing coded system has strong property $P$.


\begin{lemma}\label{lem:mod}Let $X$ be a mixing coded system, $Q$ be its generator, and $k=\gcd \{|u| : u\in Q\}$.
Then for each $u \in \lang(X)$ with $|u|= 0 \bmod{k}$ there is a word $v \in Q^+$ such that
$v=aub$ with $|a|= 0 \bmod{k}$ and $|b|= 0 \bmod{k}$.
\end{lemma}
\begin{proof}
  Since $X$ is mixing, there is a word $v\in Q^+$ such that for some $\tilde{a},\tilde{b},z_1,\ldots,z_k\in\lang(X)$ we have
  \begin{align*}
    v=\tilde{a}u z_1 u z_2 u z_3\dots u z_k \tilde{b}
  \end{align*}
  with $|z_i| = 1 \bmod{k}$ for each $i=1,\ldots,k$.
Replacing $v$ by $v^k$ if necessary, we may assume that $|v|=0 \bmod{k}$.
There is $\ell \in \{0,\dots,k-1\}$ such that $|\tilde{a}| = -\ell \bmod{k}$.
Let $a=\tilde{a}u z_1u z_2 \dots uz_\ell$
and $b=z_{\ell+1} u z_{\ell+2}\dots uz_{k}  \tilde{b}$ and observe that $|a| = 0 \bmod{k}$, $|b|=|v| -|a|-|u| = 0 \bmod{k}$
  and $v = aub$.
\end{proof}

\begin{theorem}
If $X$ is a mixing coded system, then $X$ has the strong property $P$.
\end{theorem}

\begin{proof}
Blanchard \cite[Proposition 4]{B92} proved that a shift space $X$ over $\alf$ has the property $P$
if for any integer $p$ belonging to some infinite strictly increasing sequence of integers
there exists an integer $q=q(p)$ such for any $k\ge 2$
and any words $u_1,\ldots, u_k\in \lang_p(X)$  there are
 words $w_1,\ldots, w_{k-1}\in \lang_q(X)$ such that $u_{1}w_1u_2\ldots u_{k-1}w_{k-1}u_{k}\in \lang(X)$.
Let $Q$ be a set of generators of $X$ and $k=\gcd \{|u| : u\in Q\}$.

We will show that Blanchard's criterion \cite[Proposition 4]{B92} applies to any $p\in\{\ell\in\N:\ell= 0\bmod{k}\}$.
To this end, fix $p=0 \bmod{k}$ and enumerate all blocks of length $p$ by $v_1, \ldots,v_n$.
We use Lemma \ref{lem:mod} to obtain $a_1,\dots,a_n,b_1,\dots,b_n \in \lang(X)$
such that \[0=|a_1|=\dots=|a_n|=|b_1|=\dots=|b_n| \bmod{k}\]
and $a_iv_i b_i \in Q^+$ for $i \in \{1,\ldots,n\}$.
By Remark \ref{rem:frob} there exists $N$ such that for all $n >N$ there is a word $w \in Q^+$ with $|w|=kn$, thus
we can find
$c_1,\dots,c_n\in Q^+$ and $d_1,\dots,d_n \in Q^+$ such that
$q'=|c_1a_1|=\dots=|c_na_n|$ and $q''=|b_1d_1|=\dots=|b_nd_n|$.
Now let $u_1,\dots,u_k$ be any words in $\lang_p(X)$.
Then there is a function $\phi\colon \{1,\dots,k\} \to \{1,\dots,n\}$ with
$u_i=v_{\phi(i)}$.
Therefore
\[
a_{\phi(1)}u_1 (\underbrace{b_{\phi(1)}d_{\phi(1)}c_{\phi(2)}a_{\phi(2)}}_{w_1})u_2 (\underbrace{b_{\phi(2)}d_{\phi(2)}c_{\phi(3)}d_{\phi(3)}}_{w_2})\ldots c_{\phi(n)}a_{\phi(n)}u_n b_{\phi(n)} \in Q^+.
\]
We set $q(p)=q'+q''$ and we obtain that $X$ has the strong property $P$ by \cite[Proposition 4]{B92}.
\end{proof}

\section{Two folklore results}

We finish the paper with two results which are probably folklore, but we were unable to find them in the literature so we attach them for completeness. Combining them  with Corollary \ref{cor:equiv} we obtain that the stronger forms of transitivity mentioned in the Introduction are equivalent for synchronized systems.

\begin{lemma}\label{lem:rel_prime}
Let $X$ be a coded system presented by a labelled graph $G$.
If there are two cycles on $G$ with relatively prime lengths, then $X$ is mixing.
\end{lemma}
\begin{proof}
Denote by $\alpha_1$ and $\alpha_2$  two cycles on $G$ with relatively prime lengths, $k_j=|\alpha_j|$, $j=1,2$.  Let $e_j$ be a vertex of some edge belonging to $\alpha_j$ for $j=1,2$. Let $u_j$ be the label of $\alpha_j$ for $j=1,2$ read off traversing $\alpha_j$ from $e_j$. Take any words $w_1, w_2\in \lang(X)$.
Since $X$ is coded, we can find paths $\gamma_1,\gamma_2$ on $G$ labelled, respectively, by
$w_1, w_2$. Let $a_2$ be the initial vertex of $\gamma_2$, and $b_1$ be the terminal vertex of $\gamma_1$. Let $\ell_1$ ($\ell_2$) be the length of
the shortest path $\pi_{1}$ ($\pi_2$) on $G$ from $b_1$ to $e_1$ (from $e_2$ to $a_2$) and $m$ be the length of the
the shortest path $\rho$ on $G$ from $e_1$ to $e_2$. Let $v_1,v_2,z\in\lang(X)$ be labels of $\pi_1,\pi_2,\rho$, respectively. It follows that for each $p,q\in\N$ the path
\[
\pi_{pq}=\gamma_1\pi_1(\alpha_1)^p\rho(\alpha_2)^q\pi_2
\]
is labeled by $w_1v_1(u_1)^pz(u_2)^qv_2w_2$.
Since $k_1$ and $k_2$ are relatively prime, the set
$\{pk_1+qk_2:p,q\in\N\}$ is cofinite. Therefore there is $N>0$ such that if we fix any $t\geq N$ then we can find $p=p(t),q=q(t)\in\N$ so that
$t=\ell_1 + p_{t}k_1 + m + q_{t}k_2 + \ell_2$ and the path $\pi_{p(t)q(t)}$ on $G$ has length $t$. Therefore for each $n\ge N-\ell_1$ there is a word $w$ of length $n$ such that $w_1ww_2\in\lang(X)$ and hence $X$ is mixing. \end{proof}

\begin{lemma}\label{lem:syn_rel_prime}
A synchronized shift $X$ is topologically mixing if and only if there are two closed paths with relatively prime lengths in
its Fisher cover.
\end{lemma}
\begin{proof}
The ``if'' part follows from Lemma \ref{lem:rel_prime}. For the ``only if'' part assume that $X$ is topologically mixing. Let $w$ be a synchronizing word for $X$. Then $w$ is a magic word for the Fisher cover $(G,\lab)$ of $X$, that is, there is a vertex $e$ of $G$ such that every path labelled by $w$ ends at $e$. Since $X$ is mixing there is $N\in \N$ and there are words $u_1,u_2$ with $|u_1|=N$, $|u_2|=N+1$ such that $wu_1w,wu_2w\in\lang(X)$. Because each path labelled by $w$ ends at $e$, there are closed paths in $G$ labelled by $u_1w$ and $u_2w$ with relatively prime lengths.
\end{proof}

The following theorem summarizes our results on connections between variants of transitivity for coded systems.
\begin{theorem}
Let $X$ be a non trivial coded system. Then the following conditions
are equivalent:
\begin{enumerate}[(a)]
\item\label{mix} $X$ is topologically mixing;
\item\label{pp} $X$ has the strong property $P$;
\item\label{wm} $X$ is topologically weakly mixing;
\item\label{tt} $X$ is totally transitive.
\end{enumerate}
Additionally, if $X$ is synchronized, then any of the above conditions is equivalent to
\begin{enumerate}[(a)]\setcounter{enumi}{4}
\item\label{rel-p} $X$ has two periodic points with relatively prime primary periods.
\end{enumerate}
Moreover, there exists a coded system $\X$ fulfilling \eqref{mix}--\eqref{tt}, but not \eqref{rel-p}.
\end{theorem}

\section*{Acknowledgements}

The research of P. Oprocha was supported by the Polish Ministry of Science and Higher Education from sources for science in the years 2013–2014, Grant No. IP2012 004272. The research of D.~Kwietniak was supported by the  National Science Centre (NCN) under grant no. DEC-2012/07/E/ST1/00185. 
The research of J.~Epperlein was partly supported by the German Research Foundation (DFG) through the Cluster of Excellence (EXC 1056), Center for Advancing Electronics Dresden (cfaed).

\end{document}